\documentclass[12pt]{article}
\usepackage{geometry}
\usepackage{amssymb}
\usepackage{amsmath}
\usepackage{float}
\usepackage{color}
\usepackage{graphicx}
\usepackage{subfigure}
\usepackage{tabu}
\usepackage{booktabs}
\usepackage{array}
\usepackage{caption}
\usepackage{amsthm}
\usepackage{booktabs}
\usepackage{cite}
\usepackage{indentfirst}
\numberwithin{equation}{section}
\newtheorem{theorem}{\bf Theorem}[section]
\usepackage[noend]{algpseudocode}
\usepackage{algorithmicx,algorithm}
\newtheorem{lemma}{Lemma}[section]
\newtheorem{example}{Example}[section]

\newtheorem{Definition}{Definition}[section]
\newtheorem{Corollary}{Corollary}[section]
\newcommand{\xqedhere}[2]{%
	\rlap{\hbox to#1{\hfil\llap{\ensuremath{#2}}}}}
\makeatletter
\@addtoreset{equation}{section}
\makeatother

\title{On some tensor inequalities based on the t-product}
\author{
	Zhengbang Cao\footnote{School of Mathematical Sciences, Ocean University of China, Qingdao 266100, China.
		E-Mail: {\tt caozhengbang@stu.ouc.edu.cn}},
	Pengpeng Xie\footnote{Corresponding author: School of Mathematical Sciences, Ocean University of China, Qingdao 266100, China.
		E-Mail: {\tt xie@ouc.edu.cn}.
		The work of this author is supported by the NSFC grant No. 11801534.
	}
}
\date{}
\begin{document}
	\maketitle
	\begin{abstract}
%Based on tensor functions presentded by Lund, in this paper, we concentrate the  inequalities for tensor functions. Considering the relationship between tensor fuctions and matrix $p$th roots, we derive theorems related to characterize the conditions for the existence of the tensor $p$th root. Furthermore, the inequalities for tensor $p$th root are also investigated. By the properties of the tensor spectral and Frobenius norm, we also present classical inequalities for tensor functions norm.
In this work, we investigate the tensor inequalities in the tensor t-product formalism.
The inequalities involving tensor power are proved to hold similarly as standard matrix scenarios.
We then focus on the tensor norm inequalities. The well-known arithmetic-geometric mean inequality,
H{\" o}lder inequality, and Minkowski inequality are generalized to tensors. Furthermore, we obtain some t-eigenvalue inequalities.
\\ \hspace*{\fill} 
	\end{abstract}	
\section{Introduction}
\hskip 2em Tensors are higher-order extensions of matrices. An order $m$ tensor can be regarded as a multidimensional array, which takes the form
\begin{equation*}
	\mathcal{A}=(a_{i_{1}\cdots i_{m}})\in \mathbb{R}^{n_1\times n_2\times \cdots \times n_m}.
\end{equation*}
In this paper, we mainly focus on tensors of order three.
The well-known representations of tensors are the CANDECOMP/PARAFAC \cite{Carroll1970} and Tucker models \cite{Tucker1966}.
In the last ten years, the t-product \cite{Kilmer2011,Braman2010} has been introduced as a generalization of matrix multiplication for tensors.
 It is shown that a discrete Fourier transform can be performed along the tube fibers of each tensor to compute the t-product efficiently.
The tensor operation of the t-product has been proved to be a useful tool in many
areas, finding applications in image processing \cite{Martin2013,Tarzanagh2018},
signal processing \cite{Liu2018,Sun2019}, and tensor compression \cite{Zhang2014Novel}, to name but a few.

\hskip 2em Based on the t-product framework , Lund \cite{Lund2021} posed the definition for the tensor t-function.
Later, Miao et al. \cite{Miao2020} defined the generalized tensor t-function by the tensor singular value decomposition.
After that, the authors in \cite{Zheng2021} presented the definition of the t-positive (semi)definiteness of third-order symmetric tensors.
In order to motivate further the development of the  tensor analysis, one topic we are interested in is the results specific to particular t-functions.
We will first concentrate on the tensor power including the square root of the symmetric t-positive semidefinite tensors.
The other topic of interest is the tensor t-eigenvalue.
A recent study of t-eigenvalues introduced some t-eigenvalue inequalities for Hermitian tensor
such as Weyl's theorem and Cauchy's interlacing theorem \cite{Liu_Jin2021}.

\hskip 2em As an important research field of scientific computing, matrix inequalities reflect the quantitative aspect of matrix analysis. Much work has been carried out to the development of matrix equalities \cite{Horn1991,Zhan2002}. So it is natural to talk about the tensor inequalities. In fact, Chang has built some inequalities in the aspects regarding trace function, Golden-Thompson inequality, Jenson's inequality and Klein's inequality, et al., for the t-product tensors \cite{Chang2021_1,Chang2021_2}. In this paper, we will generalize some inequalities of matrix power, matrix norm inequalities and several classical eigenvalue inequalities to tensors.

\hskip 2em This paper is organized as follows. In section 2, we review basic definitions and notations. Section 3 details the inequalities of tensor power. By the properties of the tensor spectral and Frobenius
norm, results for the norm inequalities of tensor functions are presented in section 4. Section 5 studies the tensor t-eigenvalue inequalities.  We give a conclusion in section 6.

\section{Preliminaries}
\hskip 2em In this section, we review the t-product introduced by Kilmer et al.\cite{Kilmer2013,Kilmer2011} and give some needed notations.
Throughout this paper,	third-order tensors denoted by calligraphic script letters %e.g. $\mathcal{A}$
%with real entries $a_{ijk}$
are considered. Capital letters
refer to matrices, and lower case letters to vectors. The $i$th frontal slice of tensor $\mathcal{A}$ will be denoted by $A^{(i)}$. %Entries in vectors are indexed by subscripts and
%bold face lower case letters to denote tube fibers (tubal scalars or tubes). $N(A)$ and $N(\mathcal{A})$ represent the range of matrix $A$ and tensor $\mathcal{A}$. %defined in \cite{Kilmer2013} respectively.
 %Based on the theorems we concluded above for the square root and $p$th root of tensors, now we extended some of inequalities in matrix to tensors.
For Hermitian matrices $G, H$ and symmetric tensors $\mathcal{A}$ and $\mathcal{B}$, we write $G\le H$ and $\mathcal{A}\le\mathcal{B}$ to mean that $H-G$ is Hermitian positive semidefinite and $\mathcal{B}-\mathcal{A}$ is a symmetric t-positive semidefinite tensor. In particular, $H>(\ge)0$ and $\mathcal{A}>(\ge)0$ indicate that $H$ is Hermitian positive (semi)definite and $\mathcal{A}$ is a symmetric t-positive (semi)definite tensor respectively.

\hskip 2em For $\mathcal{A}\in \mathbb{R}^{n_1\times n_2 \times n_3}$, define \texttt{bcirc} as a block circulant matrix of size $n_1n_3\times n_2n_3$
\begin{equation*}
	\texttt{bcirc}(\mathcal{A})=\begin{bmatrix}
		A^{(1)}&A^{(n_3)}&\cdots&A^{(2)}\\
		A^{(2)}&A^{(1)}&\cdots&A^{(3)}\\
		\vdots&\vdots&\ddots&\vdots\\
		A^{(n_3)}&A^{(n_3-1)}&\cdots&A^{(1)}
	\end{bmatrix}.
\end{equation*}
The command \texttt{unfold} reshapes a tensor $\mathcal{A}\in \mathbb{R}^{n_1\times n_2 \times n_3}$ into an $n_1n_3\times n_2$ block-column vector (the first block-column of $\texttt{bcirc}(\mathcal{A}))$, while \texttt{fold} is the inverse, i.e., $\texttt{fold}(\texttt{unfold}(\mathcal{A}))=\mathcal{A}$.
\begin{Definition}{(t-product) \cite{Kilmer2013}}
	Let $\mathcal{A}\in \mathbb{R}^{n_1\times n_2 \times n_3}$ and $\mathcal{B}\in \mathbb{R}^{n_2\times n_4 \times n_3}$. The t-product $\mathcal{A}*\mathcal{B}$ is the tensor $\mathcal{C}\in \mathbb{R}^{n_1\times n_4 \times n_3}$ defined by
	\begin{equation*}
		\mathcal{C}=\mathtt{fold}(\mathtt{bcirc}(\mathcal{A})\cdot \mathtt{unfold}(\mathcal{B})).
	\end{equation*}	
%where $"\cdot"$ denotes the standard matrix-matrix product.
\end{Definition}
\hskip 2em Note that the t-product reduces to the standard matrix multiplication when $n_3=1$.
The Discrete Fourier Transformation (DFT) plays a core role in tensor-tensor product. %Therefore, we give some related background knowledge and notations here.
The DFT on $v\in \mathbb{R}^n$, denoted as $\bar{v}$, is given by
$\bar{v}=F_nv\in \mathbb{C}^n.$
Here $F_n$ is the DFT matrix
\begin{equation*}
	F_n=
	\begin{bmatrix}
		1&1&1&\cdots&1\\
		1&\omega&\omega^2&\cdots&\omega^{n-1}\\
		\vdots&\vdots&\vdots&\ddots&\vdots\\
		1&\omega^{n-1}&\omega^{2(n-1)}&\cdots&\omega^{(n-1)(n-1)}
	\end{bmatrix},
\end{equation*}
where $\omega=e^{\frac{-2\pi \mathtt i}{n}}$ is a primitive $n$th root of unity with  $\texttt{i}=\sqrt{-1}$ and $F_n$ satisfies $F_n^{*}F_n=F_nF_n^{*}=nI_n.$ The block circulant matrix can be block diagonalized by the DFT, i.e.,
\begin{equation}\label{operator}
	(F_{n_3}\otimes I_{n_1})\cdot\texttt{bcirc}(\mathcal{A})\cdot(F_{n_3}^{-1}\otimes I_{n_2})=\bar{A},
\end{equation}
where $\otimes$ denotes the Kronecker product and 	$\bar{A}=\texttt{diag}(\bar{A}^{(1)},\bar{A}^{(2)},\ldots,\bar{A}^{(n_3)})$. By taking the Fast Fourier Transform (FFT) along each tubal scalar of $\mathcal{A}$, $\bar{\mathcal{A}}=\texttt{fold}(\bar{A})=\texttt{fft}(\mathcal{A},[],3)$ and $\mathcal{A}=\texttt{ifft}(\bar{\mathcal{A}},[],3).$ Then we have the following lemma.
\begin{lemma}\label{relation}\cite{Kilmer2013}
For $\mathcal{A}$, $\mathcal{B}$ and $\mathcal{C}$ of appropriate size, the following statements hold:
\begin{equation*}
\mathcal{C}=\mathcal{A}*\mathcal{B} \iff \bar{C}=\bar{A}\cdot\bar{B},
\end{equation*}
\begin{equation*}
	\mathcal{C}=\mathcal{A}+\mathcal{B} \iff \bar{C}=\bar{A}+\bar{B}.
\end{equation*}
\end{lemma}
\begin{Definition}{(identity tensor) \cite{Kilmer2013}}
	The identity tensor $\mathcal{I}\in \mathbb{R}^{n\times n\times n_3}$ is the tensor with %its first frontal slice
	$I^{(1)}$ being the $n\times n$ identity matrix, and other frontal slices being zeros.
\end{Definition}
\begin{Definition}{(tensor transpose) \cite{Kilmer2013}}
	If $\mathcal{A}\in \mathbb{R}^{n_1\times n_2\times n_3}$, then $\mathcal{A}^{\mathrm{T}}$ is the $n_2\times n_1\times n_3$ tensor obtained by transposing each of the frontal slices and then reversing the order of transposed frontal slices 2 through $n_3$.
\end{Definition}
\begin{Definition}($\mathrm{f}$-diagonal tensor)\cite{Kilmer2011}
A tensor is called $\mathrm{f}$-diagonal if each of its frontal slices is a diagonal matrix.
\end{Definition}
\begin{Definition}{(inverse tensor) \cite{Kilmer2013}}
	An $n\times n\times n_3$ tensor $\mathcal{A}$ has an inverse $\mathcal{B}$, provided that $\mathcal{A}*\mathcal{B}=\mathcal{I}_{nnn_3}$ and $ \mathcal{B}*\mathcal{A}=\mathcal{I}_{nnn_3}.$
\end{Definition}
%\begin{Definition}{(multirank)}\cite{Kilmer2013}
%	The multirank of tensor $\mathcal{A}$ is the tubal scalar $\rho=multirank(\mathcal{A})$ such that $\rho^{(i)}$ is the rank of the ith matrix $\bar{A}^{(i)}$.
%\end{Definition}
\begin{Definition}\cite{Lu2020,Kilmer2013}
	For $\mathcal{A}\in \mathbb{R}^{n_1\times n_2\times n_3}$,
	The Frobenius norm and the spectral norm of $\mathcal{A}\in \mathbb{R}^{n_1\times n_2\times n_3}$ are defined as
	\begin{equation*}
	{\left\|\mathcal{A}\right\|}_F=\sqrt{\sum_{ijk}{\left|a_{ijk}\right|}^2},\
    {\left\|\mathcal{A}\right\|}_2={\left\|\mathtt{bcirc}(\mathcal{A})\right\|}_2.
	\end{equation*}
\end{Definition}
\hskip 2em We first summarize some basic properties about tensor norms.
\begin{lemma}\label{relation_norm}\cite{Lu2020}
	Let $\mathcal{A}\in \mathbb{R}^{n\times n\times n_3}$. Then
	\begin{equation*}
		\|\mathcal{A}\|_2=\|\bar{A}\|_2,\
		\|\mathcal{A}\|_F=\frac{1}{\sqrt{n_3}}\|\bar{A}\|_F.
	\end{equation*}
\end{lemma}
%\begin{lemma}\cite{Cao2021}
%	For tensor $\mathcal{A}$ and $\mathcal{B}$ of appropriate size, the following statements hold:\\
%	(a) ${\left\|\mathcal{A}*\mathcal{B}\right\|}_F\le {\left\|\mathcal{A}\right\|}_2\cdot{\left\|\mathcal{B}\right\|}_F$, ${\left\|\mathcal{A}*\mathcal{B}\right\|}_F\le {\left\|\mathcal{A}\right\|}_F\cdot{\left\|\mathcal{B}\right\|}_2$.\\
%	(b) ${\left\|\mathcal{A}*\mathcal{B}\right\|}_2\le{\left\|\mathcal{A}\right\|}_2\cdot{\left\|\mathcal{B}\right\|}_2$,\
%	${\left\|\mathcal{A}*\mathcal{B}\right\|}_F\le \sqrt{n_3}{\left\|\mathcal{A}\right\|}_F\cdot{\left\|\mathcal{B}\right\|}_F$.\\
%	(c) ${\left\|\mathcal{A}\right\|}_2\le\sqrt{n_3}{\left\|\mathcal{A}\right\|}_F$.\\
%	(d) ${\left\|\mathcal{A}+\mathcal{B}\right\|}_2\le{\left\|\mathcal{A}\right\|}_2+{\left\|\mathcal{B}\right\|}_2$,\
%	${\left\|\mathcal{A}+\mathcal{B}\right\|}_F\le{\left\|\mathcal{A}\right\|}_F+{\left\|\mathcal{B}\right\|}_F$.
% \end{lemma}
\begin{Definition}\cite{Zheng2021}
Let $\mathcal{X},\mathcal{Y}\in \mathbb{R}^{n_1\times n_2\times n_3}$. The inner product between $\mathcal{X}$ and $\mathcal{Y}$ is defined as
\begin{equation*}
\left\langle\mathcal{X},\mathcal{Y}\right\rangle=\sum_{i,j,k}a_{ijk}b_{ijk}.
\end{equation*}
\end{Definition}
\begin{Definition}(orthogonal tensor)\cite{Kilmer2013}
An $n\times n\times n_3$ real-valued tensor $\mathcal{Q}$ is orthogonal if $\mathcal{Q}^T*\mathcal{Q}=\mathcal{Q}*\mathcal{Q}^T=\mathcal{I}$.
\end{Definition}
\begin{Definition}(symmetric t-positive (semi)definite tensor)\cite{Zheng2021}
Let $\mathcal{A}\in \mathbb{R}^{n\times n\times n_3}$. We say $\mathcal{A}$ is a symmetric t-positive (semi)definite tensor, if and only if $\mathcal{A}$ is a symmetric tensor and
\begin{equation*}
	\left\langle\mathcal{X},\mathcal{A}*\mathcal{X}\right\rangle>(\ge)0
\end{equation*}
holds for any $\mathcal{X}\in \mathbb{R}^{n\times 1\times n_3} \backslash
\left\{\mathbf{0}\right\}.$
%where $\mathcal{X}>(\ge)0$ here represent that all elements of $\mathcal{X}$ are positive (nonnegative).
\end{Definition}

\hskip 2em Note that the definition of symmetric t-positive (semi)definiteness given above is consistent with that in \cite{Kilmer2013}.
\begin{lemma}\label{relation_definite}\cite{Zheng2021}
	Suppose that $\mathcal{A}\in \mathbb{R}^{n\times n\times n_3}$ can be block diagonalized as (\ref{operator}).
	%\begin{equation*}
	%(F_{n_3}\otimes I_{n_1})\cdot\mathtt{bcirc}(\mathcal{A})\cdot(F_{n_3}^{-1}\otimes I_{n_1})=\bar(A)=\mathtt{diag}(\bar{A}^{(1)},\bar{A}^{(2)},\ldots,\bar{A}^{(n_3)}),
	%\end{equation*}
	Then $\mathcal{A}$ is a symmetric t-positive (semi)definite tensor if and only if that all the matrices $\bar{A}^{(i)}\ (i=1,\ldots,n_3)$ are Hermitian positive (semi)definite.
\end{lemma}

\hskip 2em In view of the tensor t-function \cite{Lund2021}, for any $\mathcal{A}\in \mathbb{R}^{n\times n\times n_3}$, we know that
%\begin{equation*}
%$	(F_{n_3}\otimes I_{n_1})\cdot\mathtt{bcirc}((\mathcal{A}^T*\mathcal{A})^{1/2})\cdot(F_{n_3}^{-1}\otimes I_{n_1})=(\bar{A}^H\bar{A})^{1/2},
%\end{equation*}
$(\mathcal{A}^{\mathrm{T}}*\mathcal{A})^{1/2}$ is a symmetric t-positive semidefinite tensor. Like the matrix case, we denote it as $|\mathcal{A}|$, i.e. $|\mathcal{A}|\equiv (\mathcal{A}^{\mathrm{T}}*\mathcal{A})^{1/2}$.
%\begin{lemma}(Lowner-Heninz inequality)
%If $A,B\in \mathbb{C}^{n\times n}$, sunch that $A\ge B\ge 0$and $0\le r\le1$, then
%\begin{equation*}
%	A^r\ge B^r.
%\end{equation*}
%\end{lemma}
%+++++++++++++++++++++++++++++++++++++COROLLARY 3.1
%Lemma 3.2 leads the following corollary.
The following result is a straightforward corollary of Lemma \ref{relation_definite}.
\begin{Corollary}\label{relation_dayu}	Let $\mathcal{A},\mathcal{B} \in \mathbb{R}^{n\times n\times n_3}$, then $\mathcal{A}\ge\mathcal{B}$ if and only if $\bar{A}\ge \bar{B}$.% and $\bar{A}^{(i)}\ge \bar{B}^{(i)}$.
\end{Corollary}
\begin{Definition}\cite{Liu_Jin2021}
Let $\mathcal{A}\in \mathbb{C}^{n\times n\times n_3}$. Suppose that $\mathcal{X}\in \mathbb{C}^{n\times 1\times n_3}$ and $\mathcal{X}\neq \mathbf{0}$. If
\begin{equation*}
	\mathcal{A}*\mathcal{X}=\lambda\mathcal{X},\ \lambda \in \mathbb{C},
\end{equation*}
then $\lambda$ is called a t-eigenvalue of $\mathcal{A}$ and $\mathcal{X}$ is a t-eigenvector of $\mathcal{A}$ associated to $\lambda$.
\end{Definition}
The authors in \cite{Liu_Jin2021} pointed out %that for every t-eigenvalue $\lambda$ of $\mathcal{A}$, we have
%\begin{equation*}
%	\mathtt{bcirc}(\mathcal{A})\mathtt{unfold}(\mathcal{X})=\lambda\cdot\mathtt{unfold}(\mathcal{X}),
%\end{equation*}
%i.e.,
all the t-eigenvalues of $\mathcal{A}$ are actually the eigenvalues of the matrix $\mathtt{bcirc}(\mathcal{A})$, and vice versa.
Hence the definition of the t-eigenvalues is equivalent to that given in \cite{Miao_Qi2021}.
\begin{Definition}
For $\mathcal{A}\in \mathbb{R}^{n\times n\times n_3}$, we say $\mathcal{A}$ a normal tensor if and only if
\begin{equation*}
\mathcal{A}^{\mathrm{T}}*\mathcal{A}=\mathcal{A}*\mathcal{A}^{\mathrm{T}}.
\end{equation*}
\end{Definition}
\section{Tensor power inequalities}
\hskip 2em We consider the generalization of several inequalities involving matrix power to the tensor scenarios.
\subsection{Tensor L{\" o}wner-Heinz inequality}
\hskip 2em The celebrated L{\" o}wner-Heinz inequality can be generalized to tensors as follows.
%Then we derive some classical inequalities for the $p$th roots of tensors.
%Firstly, we derive the tensor Lowner-Heninz inequality.
%+++++++++++++++++++++++++++++++++++++Th 3.1
\begin{theorem}%(L{\" o}wner-Heinz inequality)
Let $\mathcal{A},\mathcal{B} \in \mathbb{R}^{n\times n\times n_3}$, $\mathcal{A}\ge\mathcal{B}\ge 0$, $0\le r\le1$, then
\begin{equation}\label{LH}
	\mathcal{A}^{r}\ge\mathcal{B}^{r}.
\end{equation}
\end{theorem}
\begin{proof}
From Corollary \ref{relation_dayu} and the assumption, we know that $\bar{A}^{(i)}\ge \bar{B}^{(i)}\ge0$. Due to \cite[Theorem 1.1]{Zhan2002}, $(\bar{A}^{(i)})^r\ge(\bar{B}^{(i)})^r$, which means that $(\bar{A}^{(i)})^r-(\bar{B}^{(i)})^r$ is positive semidefinite.
%From Lemma \ref{relation_definite}, (\ref{LH}) holds.
The proof is complete as Lemma \ref{relation_definite}.
\end{proof}
An illustrative example below shows that this result does not hold for $r>1$.
\begin{example}
Assume $\mathcal{A}, \mathcal{B} \in \mathbb{R}^{2\times 2\times 2}$, with
\begin{equation*}
A^{(1)}=\begin{bmatrix}
		2&1\\1&1
	\end{bmatrix},
B^{(1)}=\begin{bmatrix}
	1&0\\0&0
\end{bmatrix},
\end{equation*}
and the other frontal slices of $\mathcal{A}$ and $\mathcal{B}$ are zeros. Notice that $\mathcal{A}\ge\mathcal{B}\ge 0$, however, the first frontal slice of $(\mathcal{A}^2-\mathcal{B}^2)$ is $\begin{bmatrix}
	4&3\\3&2
\end{bmatrix}$, and the second frontal slice is a zero matrix. Therefore, $\mathcal{A}^2\ge \mathcal{B}^2$ dose not hold.
\end{example}

\hskip 2em Likewise, for the tensor power, we also have the following inequalities.
%+++++++++++++++++++++++++++++++++++++++++++TH 3.2
\begin{theorem}
Let $\mathcal{Q}$ be orthogonal and $\mathcal{X}\ge0$. Then
\begin{equation}\label{power_r_1}
	\mathcal{Q}*\mathcal{X}^r*\mathcal{Q}\le(\mathcal{Q}*\mathcal{X}*\mathcal{Q})^r,\mathrm{if}\ 0<1\le 1,
\end{equation}
\begin{equation}\label{power_r_2}
	\mathcal{Q}*\mathcal{X}^r*\mathcal{Q}\ge(\mathcal{Q}*\mathcal{X}*\mathcal{Q})^r,\mathrm{if}\  1\le r\le 2.
\end{equation}
\end{theorem}
\begin{proof}
%Due to the orthogonal tenosr $\mathcal{Q}$
From the definition of orthogonality and Lemma \ref{relation}, we know that $\bar{Q}$ is a unitary matrix.
It follows from \cite[Lemma 3.1]{Zhan2002} that
\begin{equation*}
	\bar{Q}\bar{X}^r\bar{Q}\le(\bar{Q}\bar{X}\bar{Q})^r,\ 0< r\le1.
\end{equation*}
Applying Corollary \ref{relation_dayu} yields (\ref{power_r_1}). (\ref{power_r_2}) can be proved similarly.
\end{proof}
%+++++++++++++++++++++++++++++++++++++++++++TH 3.3
\begin{theorem}\label{power_ineq}
If $\mathcal{A}\ge\mathcal{B}\ge0$, then
\begin{equation}\label{power}
	(\mathcal{B}^r*\mathcal{A}^p*\mathcal{B}^r)^{1/q}\ge\mathcal{B}^{(p+2r)/q}
\end{equation}
and
\begin{equation}\label{power_cf2}
	\mathcal{A}^{(p+2r)/q}\ge(\mathcal{A}^r*\mathcal{B}^p*\mathcal{A}^r)^{1/q}
\end{equation}
for $r\ge0,p\ge0,q\ge1$ with $(1+2r)q\ge p+2r$.
\end{theorem}
\begin{proof}
The inequality (\ref{power}) is a direct consequence of Corollary \ref{relation_dayu} and
\begin{equation*}
(\bar{B}^r\cdot\bar{A}^p\cdot\bar{B}^r)\ge\bar{B}^{(p+2r)/q},
\end{equation*}
which follows from \cite[Theorem 1.16]{Zhan2002}. We can also obtain (\ref{power_cf2}) in the same way.
\end{proof}

\hskip 2em Notice that the case $p=q\ge1$ of Theorem \ref{power_ineq} is the following.
%++++++++++++++++++++++++++++++COROLLARY 3.2
\begin{Corollary}
	Suppose $\mathcal{A}\ge \mathcal{B}\ge 0.$ We obtain that
	\begin{equation*}
		(\mathcal{B}^r*\mathcal{A}^p*\mathcal{B}^r)^{1/p}\ge\mathcal{B}^{(p+2r)/p}\ and\
	\mathcal{A}^{(p+2r)/p}\ge(\mathcal{A}^r*\mathcal{B}^p*\mathcal{A}^r)^{1/p}
\end{equation*}
hold for all $r\ge0$ and $p\ge1$. Especially, if $r=1,p=2$, then
	\begin{equation*}
		(\mathcal{B}*\mathcal{A}^2*\mathcal{B})^{1/2}\ge \mathcal{B}^2 \ and \
		\mathcal{A}^2\ge (\mathcal{A}*\mathcal{B}^2*\mathcal{A})^{1/2}.
	\end{equation*}
\end{Corollary}
\subsection{Tensor Young inequality}
\hskip 2em The gist of the most important case of the Young inequality is that
 if $1/p + 1/q = 1$, with $p,q>1$, then $|ab|\le{|a|}^p/p+{|b|}^q/q$ for $a,b \in \mathbb{C}$.
 Ando  in \cite{Ando1995} pointed out that if $A,B$ is a commuting pair and $AB\ge 0$, then it is clear that
\begin{equation*}
	AB\le\frac{A^p}{p}+\frac{B^q}{q}.
\end{equation*}
Now we extend this classical inequality to tensors.
%++++++++++++++++++++++++++++TH 3.4
\begin{theorem}%{(Tensor Young Inequality)}
If $\mathcal{A},\mathcal{B}\ge0$ is a commuting pair, i.e., $\mathcal{A}*\mathcal{B}=\mathcal{B}*\mathcal{A}$ and $\mathcal{A*B}\ge 0$, then for $p,q>1$,
\begin{equation*}
	\mathcal{A}*\mathcal{B}\le \frac{1}{p}\mathcal{A}^p+\frac{1}{q}\mathcal{B}^q.
\end{equation*}
%where $\frac{1}{p}\mathcal{I}$ is a tensor whose first frontal slice is $\frac{1}{p}I$ and others are zeros, and $1/p+1/q=1$.
\begin{proof}
Begin with the fact that $\mathrm{bcirc}(\mathcal{A}+\mathcal{B})=\mathrm{bcirc}(\mathcal{A})+\mathrm{bcirc}(\mathcal{B})$. Then we have
\begin{equation*}
\begin{aligned}
&(F_{n_3}\otimes I_{n_1})\cdot\mathrm{bcirc}(\mathcal{A}*\mathcal{B})\cdot(F_{n_3}^{-1}\otimes I_{n_1})=\bar{A}\bar{B}\le\frac{1}{p}\bar{A}^p+\frac{1}{q}\bar{B}^q\\=&
(F_{n_3}\otimes I_{n_1})\cdot\mathrm{bcirc}(\frac{1}{p}\mathcal{A}^p+\frac{1}{q}\mathcal{B}^q)\cdot(F_{n_3}^{-1}\otimes I_{n_1}).
\end{aligned}
\end{equation*}
%where we exploit the fact that $\mathrm{bcirc}(\mathcal{A}+\mathcal{B})=\mathrm{bcirc}(\mathcal{A})+\mathrm{bcirc}(\mathcal{B})$.
Applying Corollary \ref{relation_dayu} again gives the conclusion.
\end{proof}
\end{theorem}
 %\hskip 2em Based on properties we've had above, we can also generalize the conclusions of the Young inequality to tensors.
\hskip 2em A generalized matrix Young inequality was given in \cite{Ando1995}.
 %+++++++++++++++++++++++++++++Lemma 3.1
\begin{lemma}\label{Young_lemma}\cite{Ando1995}
Let $p,q>0$ be mutually conjugate exponents, that is, $1/p+1/q=1.$ Then for any pair $A,B$ of $n\times n$ complex matrices, there is a unitary matrix $U$
depending on $A,B$ such that
\begin{equation*}
U^H|AB^H|U\le\frac{|A|^p}{p}+\frac{|B|^q}{q}.
\end{equation*}
\end{lemma}
\hskip 2em It turns out that the generalized tensor Young inequality can be stated below.%As a classic result of Young inequality, now extend Lemma 3.3 to tensors.
%++++++++++++++++++++++++++++++++++++++TH 3.5
\begin{theorem}
Let $p,q>0$, $1/p+1/q=1$. Then for any pair $\mathcal{A},\mathcal{B}\in \mathbb{R}^{n\times n\times n_3},$
there exists an orthogonal tensor $\mathcal{U}$ depending on $\mathcal{A}$ and $\mathcal{B}$, such that
\begin{equation*}\label{Ando}
	\mathcal{U}^T*|\mathcal{A}*\mathcal{B}^T|*\mathcal{U}\le\frac{1}{p}|\mathcal{A}|^p+\frac{1}{q}|\mathcal{B}|^q.
\end{equation*}
\end{theorem}
\begin{proof}
First, we apply (\ref{operator}) to $\mathcal{A},\mathcal{B}$ and then we get $\bar{A}$ and $\bar{B}$,
which are block diagonal matrices. According to Lemma \ref{Young_lemma}, for any pair $\bar{A}^{(i)},\bar{B}^{(i)}$, there is a unitary matrix $U^{(i)}$ such that
\begin{equation*}
(U^{(i)})^H|\bar{A}^{(i)}(\bar{B}^{(i)})^H|U^{(i)}\le|\bar{A}^{(i)}|^p/p+|\bar{B}^{(i)}|^q/q.
\end{equation*}
Set
\begin{equation*}
\mathcal{\bar{U}}=\mathtt{fold}\left(\begin{bmatrix}
	U^{(1)}\\U^{(2)}\\\vdots\\ U^{(n_3)}
\end{bmatrix}\right), \mathcal{U}=\mathtt{ifft}\left(\bar{\mathcal{U}},[\ ],3\right),
\end{equation*}
and
\begin{equation*}
	\mathcal{\bar{C}}=\mathtt{fold}\left(\begin{bmatrix}
		|\bar{A}^{(1)}(\bar{B}^{(1)})^H|\\|\bar{A}^{(2)}(\bar{B}^{(2)})^H|\\\vdots\\ |\bar{A}^{(n_3)}(\bar{B}^{(n_3)})^H|
	\end{bmatrix}\right), \mathcal{C}=\mathtt{ifft}\left(\bar{C},[\ ],3\right).
\end{equation*}
Obviously, $\mathcal{U}$ is an orthogonal tensor and $\mathcal{C}=|\mathcal{A}*\mathcal{B}^T|$. Analogously, we can get $|\mathcal{A}|^p$ and  $|\mathcal{B}|^q$.
\end{proof}
\hskip 2em In particular, when we take $p=q=2$, we get a corollary which is a tensor generalization of the result obtained by Bhatia and Kittaneh in \cite{Bhatia1990}.
%++++++++++++++++++++COROLLARY 3.2
\begin{Corollary}
For any pair $\mathcal{A},\mathcal{B}\in \mathbb{R}^{n\times n\times n_3}$, there is an orthogonal tensor $\mathcal{U}$ depending on $\mathcal{A}$ and $\mathcal{B}$, such that
\begin{equation*}
	\mathcal{U}^T*|\mathcal{A}*\mathcal{B}^T|*\mathcal{U}\le\frac{1}{2}|\mathcal{A}|^2+\frac{1}{2}|\mathcal{B}|^2.
\end{equation*}
\end{Corollary}
\section{Tensor norm inequalities}
\hskip 2em In this section, we derive some tensor norm inequalities based on the Frobenius norm and the spectral norm. Particularly, we will prove some norm inequalities related to the tensor power.

\hskip 2em Notice that any tensor $\mathcal{T}\in\mathbb{C}^{n\times n\times n_3}$ can be written as $\mathcal{T}=\mathcal{A}+\texttt{i}\mathcal{B}$, where $\mathcal{A},\mathcal{B}\in\mathbb{R}^{n\times n\times n_3}$, and $(\mathtt{i}\mathcal{B})(i,j,k)=\mathtt{i}\mathcal{B}(i,j,k).$ Then we have
\begin{equation*}
	\begin{aligned}
		&(F_{n_3}\otimes I_{n})\cdot\mathtt{bcirc}(\mathcal{T})\cdot(F_{n_3}^{-1}\otimes I_{n})\\=&(F_{n_3}\otimes I_{n})\cdot\mathtt{bcirc}(\mathcal{A})\cdot(F_{n_3}^{-1}\otimes I_{n})+
		(F_{n_3}\otimes I_{n})\cdot\mathtt{bcirc}(\texttt{i}\mathcal{B})\cdot(F_{n_3}^{-1}\otimes I_{n})\\
		=&\bar{A}+\mathtt{i}(F_{n_3}\otimes I_{n})\cdot\mathtt{bcirc}(\mathcal{B})\cdot(F_{n_3}^{-1}\otimes I_{n})\\
		=&\bar{A}+\mathtt{i}\bar{B}.
	\end{aligned}
\end{equation*}
%which means that apply (\ref{operator}) to $\mathcal{T}$ defined above, we get $\bar{T}=\bar{A}+i\bar{B}$ correspondingly.
%Based on this property, we have the following theorem.
Before we move on to several classical inequalities, we mention some inequalities for the complex tensors.
%+++++++++++++++++++++++++++++TH 4.1
\begin{theorem}
Let $\mathcal{T}=\mathcal{A} + \mathtt i\mathcal{B}$, $\mathcal{A},\mathcal{B}\in\mathbb{R}^{n\times n\times n_3}$. \\
(a) if $\mathcal{A}$ and $\mathcal{B}$ are symmetric, then
\begin{equation*}
 (\|\mathcal{A}\|_2^2+\|\mathcal{B}\|_2^2)\le\|\mathcal{T}\|_2^2\le2(\|\mathcal{A}\|_2^2+\|\mathcal{B}\|_2^2),
\end{equation*}
\begin{equation*}
	4(\|\mathcal{A}\|_F^2+\|\mathcal{B}\|_F^2)\ge\|\mathcal{T}\|_F^2\ge(\|\mathcal{A}\|_F^2+\|\mathcal{B}\|_F^2),
\end{equation*}
\begin{equation*}
	\|(\mathcal{A}^2+\mathcal{B}^2)^{1/2}\|_2\le\|T\|_2\le\sqrt{2}\|(\mathcal{A}^2+\mathcal{B}^2)^{1/2}\|_2,
\end{equation*}
\begin{equation*}
	\|(\mathcal{A}^2+\mathcal{B}^2)^{1/2}\|_F=\|T\|_F.
\end{equation*}

(b) if $\mathcal{A}$ is symmetric t-positive semidefinite and $\mathcal{B}$ is symmetric, then
\begin{equation*}\label{norm_eq_1}
	\|\mathcal{T}\|^2_2\le\|\mathcal{A}\|^2_2+2\|\mathcal{B}\|^2_2,\
	\|\mathcal{T}\|^2_F\ge\|\mathcal{A}\|^2_F+2\|\mathcal{B}\|^2_F.
\end{equation*}

(c) if $\mathcal{A}$ and $\mathcal{B}$ are symmetric t-positive semidefinite, then
\begin{equation*}
	\|\mathcal{T}\|^2_2\le\|\mathcal{A}\|^2_2+\|\mathcal{B}\|^2_2,\
	\|\mathcal{T}\|^2_F\le\|\mathcal{A}\|^2_F+\|\mathcal{B}\|^2_F.
\end{equation*}
\end{theorem}
\begin{proof}
The proof makes use of \cite[Theorem 3.21]{Zhan2002} and Lemma \ref{relation_norm}. Indeed,
\begin{equation*}
\begin{aligned}
&4(\|\mathcal{A}\|_F^2+\|\mathcal{B}\|_F^2)=\frac{4}{n_3}(\|\bar{A}\|_F^2+\|\bar{B}\|_F^2)\\
\ge&\frac{1}{n_3}\|\bar{T}\|_F^2=\|\mathcal{T}\|_F^2\ge\frac{1}{n_3}(\|\bar{A}\|_F^2+\|\bar{B}\|_F^2)=(\|\mathcal{A}\|_F^2+\|\mathcal{B}\|_F^2),
\end{aligned}
\end{equation*}
and
\begin{equation*}
\begin{aligned}
&\|(\mathcal{A}^2+\mathcal{B}^2)^{1/2}\|_2=\|(\bar{A}^2+\bar{B}^2)^{1/2}\|_2\\\le&\|\bar{T}\|_2=\|T\|_2\le\sqrt{2}\|(\bar{A}^2+\bar{B}^2)^{1/2}\|_2=\sqrt{2}\|(\mathcal{A}^2+\mathcal{B}^2)^{1/2}\|_2.
\end{aligned}
\end{equation*}
%hence, (\ref{T=A+iB_1}) holds. Combing with Lemma 5.1 and 5.2, the spectral norm inequality can be proved.
A similar procedure can be used for (b) and (c).
\end{proof}

\hskip 2em We are now set to state three classical inequalities.
The arithmetic-geometric mean inequality for complex numbers $a, b$ is $|ab|\le(|a|^2+|b|^2)/2$. One tensor version of this inequality is the following result.
%\begin{lemma}\cite{Zhan2002}
%For any three matrices $A, B, X$ we have
%\begin{equation*}
%	\|AXB^H\|\le\|A^HAX+XB^HB\|
%\end{equation*}
%for every unitarily invariant norm.
%\end{lemma}
%\hskip 2em Now we extend it to tensors.
%%============================================TH 4.2
\begin{theorem}%{(Tensor arithmetic-geometric mean inequality)}
For any real tensors $\mathcal{A}$, $\mathcal{X}$ and $\mathcal{B}$ of appropriate size, we have
\begin{equation*}
\|\mathcal{A}*\mathcal{X}*\mathcal{B}^T\|\le\frac{1}{2}\|\mathcal{A}^T*\mathcal{X}+\mathcal{X}*\mathcal{B}^T*\mathcal{B}\|
\end{equation*}
for the Frobenius norm and the spectral norm.
\end{theorem}
\begin{proof}
Upon consideration of \cite[Theorem 4.19]{Zhan2002} and Lemma \ref{relation_norm}, we see that
%From \cite[Theorem 4.19]{Zhan2002} and Lemma \ref{relation_norm}, we have
\begin{equation*}
\begin{aligned}
&\|\mathcal{A}*\mathcal{X}*\mathcal{B}^T\|_F=\frac{1}{\sqrt{n_3}}\|\bar{A}\bar{X}\bar{B}^H\|_F\\\le&\frac{1}{2\sqrt{n_3}}\|\bar{A}^H\bar{A}\bar{X}+\bar{X}\bar{B}^H\bar{B}\|_F
=\frac{1}{2}\|\mathcal{A}^T*\mathcal{X}+\mathcal{X}*\mathcal{B}^T*\mathcal{B}\|_F.
\end{aligned}
\end{equation*}
The proof of the spectral norm case is entirely analogous.
\end{proof}

%For t-positive (semi)definite tensors, we have the following tensor norm inequalities.
If, additionally, the tensors $\mathcal{A},\mathcal{X}$ and $\mathcal{B}$ are symmetric t-positive definite, then the following theorem arises by
combining \cite[Theorems 4.24 and 4.25]{Zhan2002} and Lemma \ref{relation_norm}.
%============================TH 4.3
\begin{theorem}
Let $ \mathcal{A},\mathcal{X},\mathcal{B}\in\mathbb{R}^{n\times n\times n_3}$ with $ \mathcal{A}$ and $\mathcal{B}$ symmetric t-positive semidefinite. Then for the Frobenius norm and the spectral norm, the following two inequalities hold.\\
(1) For any real numbers $r, t$ satisfying $1 \le2r \le 3$, $-2<t\le2$,
\begin{equation*}
	(2+t)\|\mathcal{A}^r*\mathcal{X}*\mathcal{B}^{2-r}+\mathcal{A}^{2-r}*\mathcal{X}*\mathcal{B}^r\|\le2\|\mathcal{A}^2*\mathcal{X}+t\mathcal{A}*\mathcal{X}*\mathcal{B}+\mathcal{X}*\mathcal{B}^2\|.
\end{equation*}
(2)
\begin{equation*}
	4\|\mathcal{A}*\mathcal{B}\|\le\|(\mathcal{A}+\mathcal{B})^2\|.
\end{equation*}
\end{theorem}
%\begin{proof}
%These results are immediate from \cite[Theorem 4.24]{Zhan2002}, \cite[Theorem 4.25]{Zhan2002} and Lemma \ref{relation_norm}.
%\end{proof}

%\begin{proof}
%According to \cite[Theorem 4.24]{Zhan2002},
%\begin{equation*}
%\begin{aligned}
%&(2+t)\|\mathcal{A}^r*\mathcal{X}*\mathcal{B}^{2-r}+\mathcal{A}^{2-r}*\mathcal{X}*\mathcal{B}^r\|_F=\frac{2+t}{\sqrt{n_3}}\|\bar{A}^r\bar{X}\bar{B}^{2-r}+\bar{A}^{2-r}\bar{X}\bar{B}^r\|_F\\\le&\frac{2}{n_3}\|\bar{A}^2\bar{X}+t\bar{A}\bar{X}\bar{B}+\bar{X}\bar{B}^2\|_F
%=2\|\mathcal{A}^2*\mathcal{X}+t\mathcal{I}*\mathcal{A}*\mathcal{X}*\mathcal{B}+\mathcal{X}*\mathcal{B}^2\|_F.
%\end{aligned}
%\end{equation*}
%Similar manipulation gives the spectral norm inequality.
%\end{proof}
%\begin{theorem}
%Let $ \mathcal{A},\mathcal{X}\in\mathbb{R}^{n\times n\times n_3}$ be positive semidefinite. Then
%\begin{equation*}
%4\|\mathcal{A}*\mathcal{B}\|\le\|(\mathcal{A}+\mathcal{B})^2\|.
%\end{equation*}
%for the Frobenius norm and spectral norm.
%\end{theorem}
%\begin{proof}
%Combining with \cite[Theorem 4.25]{Zhan2002}, the theorem can be proved by the similar menthod as in Theorem 5.9.
%\end{proof}
\hskip 2em Now we turn to the tensor H{\" o}lder inequality.
%++++++++++++++++++++++++++++++++TH 4.4
\begin{theorem}%(Tensor H{\" o}lder inequality)
Let $ \mathcal{A},\mathcal{X}$ and $\mathcal{B}\in\mathbb{R}^{n\times n\times n_3}$ with $\mathcal{A}$ and $\mathcal{B}$ being symmetric t-positive semidefinite. Then
\begin{equation*}
	\left\|\ \left|\mathcal{A}*\mathcal{X}*\mathcal{B}\right|\ \right\|_F\le n_3^{\frac{1}{2p}+\frac{1}{2q}-\frac{1}{2}}\|\ |\mathcal{A}^p*\mathcal{X}|^r\ \|_F^{1/p}
\cdot \|\ |\mathcal{X}*\mathcal{B}^q|^r\ \|_F^{1/q},
\end{equation*}
and
\begin{equation*}
\left\|\ |\mathcal{A}*\mathcal{X}*\mathcal{B}|\ \right\|_2\le \|\ |\mathcal{A}^p*\mathcal{X}|^r\ \|_2^{1/p}
\cdot \|\ |\mathcal{X}*\mathcal{B}^q|^r\ \|_2^{1/q},
\end{equation*}
for all positive real numbers $r, p, q $ with $1/p+1/q=1$.
\end{theorem}
\begin{proof}
By \cite[Theorem 4.29]{Zhan2002}, it is easy to verify that
\begin{equation*}
\begin{aligned}
&\left\|\ |\mathcal{A}*\mathcal{X}*\mathcal{B}|\ \right\|_F=\frac{1}{\sqrt{n_3}}\left\|\ |\bar{A}\bar{X}\bar{B}|\ \right\|_F\\
\le&\frac{1}{\sqrt{n_3}}\left(\left\|\ |\bar{A}^p\bar{X}|^r\ \right \|^{1/p}_F\cdot\left\|\ |\bar{X}\bar{B}^q|^r\ \right\|^{1/q}_F\right)\\
=&n_3^{\frac{1}{2p}+\frac{1}{2q}-\frac{1}{2}}\left\|\ |\mathcal{A}^p*\mathcal{X}|^r\ \right\|_F^{1/p}
\cdot \|\ |\mathcal{X}*\mathcal{B}^q|^r\ \|_F^{1/q}.
\end{aligned}
\end{equation*}
The second inequality follows by mimicking the above argument.
\end{proof}
%\hskip 2em While combining with \cite[Corollary 4.27]{Zhan2002}, we can prove the following tensor norm inequality in the same way as Theorem 5.11.
%==============================COROLLARY 4.1
\begin{Corollary}
For $ \mathcal{A},\mathcal{B}\in\mathbb{R}^{n\times n\times n_3}$, then
\begin{equation*}
\left\|\ |\mathcal{A}*\mathcal{B}|^r\ \right\|_F\le n_3^{{\frac{1}{2p}+\frac{1}{2q}-\frac{1}{2}}}\left\|\ |\mathcal{A}|^{pr}\ \right\|_F^{1/p}\cdot\left\|\ |\mathcal{B}|^{qr}\ \right\|_F^{1/q},
\end{equation*}
and
\begin{equation*}
	\left\|\ |\mathcal{A}*\mathcal{B}|^r\ \right\|_2\le \left\|\ |\mathcal{A}|^{pr}\ \right\|_2^{1/p}\cdot\left\|\ |\mathcal{B}|^{qr}\ \right\|_2^{1/q},
\end{equation*}
where $r, p, q$ are positive real numbers with $1/p+1/q=1$.
\end{Corollary}
The subsequent result is another tensor H{\" o}lder inequality.
%===================================TH 4.5
\begin{theorem}
Let $1 \le p, q \le \infty$ with $1/p + 1/q= 1$. Then for all $\mathcal{A},\mathcal{B},\mathcal{C},\mathcal{D}\in\mathbb{R}^{n\times n\times n_3}$, we have
\begin{equation*}
2^{-|\frac{1}{p}-\frac{1}{2}|}\left\|\mathcal{C}^T*\mathcal{A}+\mathcal{D}^T*\mathcal{B}\right\|_F\le
n_3^{\frac{1}{2p}+\frac{1}{2q}-\frac{1}{2}}
\left\|\ |\mathcal{A}|^p+|\mathcal{B}^p|\ \right\|_F^{1/p}\cdot\left\|\ |\mathcal{C}|^q+|\mathcal{D}|^q\ \right\|^{1/q}_F,
\end{equation*}
and
\begin{equation*}
	2^{-|\frac{1}{p}-\frac{1}{2}|}\left\|\mathcal{C}^T*\mathcal{A}+\mathcal{D}^T*\mathcal{B}\right\|_2\le
	\left\|\ |\mathcal{A}|^p+|\mathcal{B}^p|\ \right\|_2^{1/p}\cdot\left\|\ |\mathcal{C}|^q+|\mathcal{D}|^q\ \right\|^{1/q}_2.
\end{equation*}
\end{theorem}
\begin{proof}
We only consider the Frobenius norm case. It is straightforward to show that
\begin{equation*}
\begin{aligned}
&2^{-|\frac{1}{p}-\frac{1}{2}|}\left\|\mathcal{C}^T*\mathcal{A}+\mathcal{D}^T*\mathcal{B}\right\|_F
=\frac{1}{\sqrt{n_3}}2^{-|\frac{1}{p}-\frac{1}{2}|}\left\|\bar{C}^H\bar{A}+\bar{D}^H\bar{B}\right\|_F\\
\le&\frac{1}{\sqrt{n_3}}\left\|\ |\bar{A}|^p+|\bar{B}^p|\ \right\|_F^{1/p}\cdot\left\|\ |\bar{C}|^q+|\bar{D}|^q\ \right\|^{1/q}_F\\
=&n_3^{\frac{1}{2p}+\frac{1}{2q}-\frac{1}{2}}
\left\|\ |\mathcal{A}|^p+|\mathcal{B}^p|\ \right\|_F^{1/p}\cdot\left\|\ |\mathcal{C}|^q+|\mathcal{D}|^q\ \right\|^{1/q}_F,
\end{aligned}
\end{equation*}
where we utilize \cite[Theorem 4.34]{Zhan2002}. %We omit the proof of the spectral norm.
\end{proof}

%\hskip 2em Then we derive the matrix Minkowski inequality to tensors.
%\begin{lemma}\cite{Zhan2002}(matrix Minkowski inequality.)
%Let $1\le p\le \infty$. For $A_i,B_i\in\mathbb{C}^{n\times n}(i=1,2)$ and every
%unitarily invariant norm,
%\begin{equation*}
%	\begin{aligned}
%		&2^{-|\frac{1}{p}-\frac{1}{2}|}\|\ |A_1+A_2|^p+|B_1+B_2|^p\ \|^{1/p}\\
%		\le&\|\ |A_1|^p+|B_1|^p\ \|^{1/p}+\|\ |A_2|^p+|B_2|^p\ \|^{1/p}.
%	\end{aligned}
%\end{equation*}
%\end{lemma}
\hskip 2em %The next result is a tensor Minkowski inequality.
We continue in this section by discussing the tensor Minkowski inequality.
\begin{theorem}%(Tensor Minkowski inequality)
Let $1\le p\le \infty$. For $ \mathcal{A}_i$ and $\mathcal{B}_i\in\mathbb{R}^{n_1\times n_1\times n_3}(i=1,2)$,
\begin{equation*}
%\begin{aligned}
	2^{-|\frac{1}{p}-\frac{1}{2}|}\left\|\ |\mathcal{A}_1+\mathcal{A}_2|^p+|\mathcal{B}_1+\mathcal{B}_2|^p\ \right\|^{1/p}\\
	\le\left\|\ |\mathcal{A}_1|^p+|\mathcal{B}_1|^p\ \right\|^{1/p}+\left\|\ |\mathcal{A}_2|^p+|\mathcal{B}_2|^p\ \right\|^{1/p},
%\end{aligned}
\end{equation*}
holds for both the Frobenius norm and the spectral norm.
\end{theorem}
\begin{proof}
%To show that the Frobenius norm case holds, we use \cite[Theorem 4.35]{Zhan2002}. It follows that
It follows from \cite[Theorem 4.35]{Zhan2002} that
\begin{equation*}
	\begin{aligned}
		&2^{-|\frac{1}{p}-\frac{1}{2}|}\|\ |\mathcal{A}_1+\mathcal{A}_2|^p+|\mathcal{B}_1+\mathcal{B}_2|^p\ \|_F^{1/p}\\
		=&n_3^{-\frac{1}{2p}}2^{-|\frac{1}{p}-\frac{1}{2}|}\|\ |\bar{A}_1+\bar{A}_2|^p+|\bar{B}_1+\bar{B}_2|^p\ \|_F^{1/p}\\
		\le&n_3^{-\frac{1}{2p}}\|\ |\bar{A}_1|^p+|\bar{B}_1|^p\ \|_F^{1/p}     +          \|\ |\bar{A}_2|^p+|\bar{B}_2|^p\ \|_F^{1/p}\\
		=&\|\ |\mathcal{A}_1|^p+|\mathcal{B}_1|^p\ \|_F^{1/p}      +         \|\ |\mathcal{A}_2|^p+|\mathcal{B}_2|^p\ \|_F^{1/p}.
	\end{aligned}
\end{equation*}
The spectral norm inequality arises analogously.
\end{proof}
%\begin{remark}
%In addition to some of the commonly used inequalities of tensor function norms, through Lemma 5.1 and 5.2, we also extend the classical norm inequalities such as  Minkowski inequality and arithmetic-Geometric mean inequality. It is worth noting that in the derivation of these inequalities, we do not take the approach similar to that in the matrix cases, but take the transformation relationship between the matrix norms and the tensor norms, which will greatly simplify our derivation and result in a more concise forms.
%\end{remark}
%=======================================SEC 5
%=======================================SEC 5
%=======================================SEc 5
\section{Tensor t-eigenvalue inequalities}

\hskip 2em Let $A$ be an $n\times n$ matrix, with eigenvalues $\lambda_1,\ldots, \lambda_n$. The Schur inequality\cite{Schur1994} says that
\begin{equation}\label{schur_1}
	\sum_{i=1}^n|\lambda_i|^2\le\|A\|_F^2.
\end{equation}
%is well known.
Based on the concepts of the tensor t-eigenvalue and Lemma \ref{relation_norm}, we can easily extend (\ref{schur_1}) to tensors.
\begin{theorem}%(Tensor Schur inequality)
	For $ \mathcal{A}\in\mathbb{R}^{n\times n\times n_3}$ with t-eigenvalues $\lambda_i,\ i=1,2,\ldots,nn_3$, we have
	\begin{equation*}
		\sum_{i=1}^{nn_3}|\lambda_i|^2\le n_3\|\mathcal{A}\|_F^2.
	\end{equation*}
\end{theorem}
%$Proof.$
%	Recalling (\ref{schur_1}) and the property of the t-eigenvalue, we obtain
%	\begin{equation*}
%		\sum_{i=1}^{nn_3}|\lambda_i|^2\le\|\mathtt{bcirc}(\mathcal{A})\|_F^2=n_3\|\mathcal{A}\|_F^2. \xqedhere{118.5pt}{\qed}
%	\end{equation*}

\hskip 2em Suppose $A$ is an $n\times n$ matrix. The Gershgorin circle theorem makes the observation:  discs $G(A)=\bigcup\limits_{i=1}^n\{z\in\mathbb{C}:|z-a_{ii}|\le\sum\limits_{j\neq i}|a_{ij}|\}$ that are centered at the points $a_{ii}$ are guaranteed to contain the eigenvalues of $A$. The upcoming theorem shows that we still have similar results in tensors.
\begin{theorem}%(Tensor Gershgorin circle theorem)
Let $\mathcal{A}=[a_{ijk}]\in\mathbb{C}^{n\times n\times n_3}$ and denote the $n$ Gershgorin discs
\begin{equation*}
G_i(\mathcal{A})=\left\{z\in\mathbb{C}:|z-a_{ii1}|\le\sum\limits_{j,k\neq i}|a_{ijk}|\right\}.
\end{equation*}
Then the t-eigenvalues of $\mathcal{A}$ are in the union of Gershgorin discs
\begin{equation*}
	G(\mathcal{A})=\bigcup\limits_{i=1}^nG_i(\mathcal{A}).
\end{equation*}
Furthermore, if the union of $k$ of the $n$ discs that comprise $G(\mathcal{A})$ forms a set that is disjoint from the remaining $n-k$ discs,
then this set contains exactly $k$ t-eigenvalues of $\mathcal{A}$, counted according to their algebraic multiplicities.
\begin{proof}
Suppose that $\mathtt{bcirc}(\mathcal{A})=[b_{ij}]_{nn_3\times nn_3}$. Since all the t-eigenvalues of $\mathcal{A}$ can be regarded as those of matrix $\mathtt{bcirc}(\mathcal{A})$, they are contained in
\begin{equation*}
G_i(\mathtt{bcirc}(\mathcal{A}))=\left\{z\in\mathbb{C}:|z-b_{ii}|\le\sum\limits_{j\neq i}|b_{ij}|\right\},\ i=1,2,\ldots,nn_3,
\end{equation*}
which can also be characterized below by noticing the block circulant structure of $\mathtt{bcirc}(\mathcal{A}),$
\begin{equation*}
%\begin{aligned}
%&G_i(\mathtt{bcirc}(\mathcal{A}))=\left\{z\in\mathbb{C}:|z-b_{ii}|\le\sum\limits_{j\neq i}|b_{ij}|\right\}\\
G_i(\mathtt{bcirc}(\mathcal{A}))=\left\{z\in\mathbb{C}:|z-a_{ii1}|\le\sum\limits_{j,k\neq i}|a_{ijk}|\right\}=G_i(\mathcal{A}),\ i=1,2,\ldots n.
%\end{aligned}
\end{equation*}
%Considering the Gershgorin circle theorem for matrices gives the conclusion.
The conclusion follows directly by the matrix Gershgorin circle theorem.
\end{proof}
\end{theorem}

\hskip 2em The next theorem is the tensor Bauer-Fike theorem.
\begin{theorem}%(Tensor Bauer-Fike theorem)
Let $\mathcal{A},\mathcal{B}\in\mathbb{C}^{n\times n\times n_3}$ be diagonalizable, and suppose that $\mathcal{A}=\mathcal{Q}^{-1}*\mathcal{S}*\mathcal{Q}$, in which $\mathcal{S}$ is $\mathrm{f}$-diagonal and $\mathcal{Q}$ is invertible. If $\lambda$ is a t-eigenvalue of $\mathcal{A}$, there always exists a $\mu$ being a t-eigenvalue of $\mathcal{B}$, such that
\begin{equation*}
|\lambda-\mu|\le\left\|\mathcal{Q}^{-1}\right\|_2\cdot \left\|\mathcal{Q}\right\|_2\cdot\left\|\mathcal{A-B}\right\|_2.
\end{equation*}
\end{theorem}
\begin{proof}
From $\mathcal{A}=\mathcal{Q}^{-1}*\mathcal{S}*\mathcal{Q}$, we have
\begin{equation*}
\mathtt{bcirc}(\mathcal{A})=\mathtt{bcirc}(\mathcal{Q}^{-1}*\mathcal{S}*\mathcal{Q})=(\mathtt{bcirc}(\mathcal{Q}))^{-1}\cdot\mathtt{bcirc}(\mathcal{S})\cdot\mathtt{bcirc}(\mathcal{Q}),
\end{equation*}
where $\mathtt{bcirc}(\mathcal{S})$ is a diagonal matrix. According to the Bauer-Fike theorem for matrices, for any eigenvalue $\lambda$ of $\mathtt{bcirc}(\mathcal{A})$, there exists an eigenvalue $\mu$ of $\mathtt{bcirc}(\mathcal{B})$, such that
\begin{equation*}
\begin{aligned}
&|\lambda-\mu|\le\left\|\mathtt{bcirc}(\mathcal{Q})^{-1}\right\|_2\cdot \left\|\mathtt{bcirc}(\mathcal{Q})\right\|_2\cdot\left\|\mathtt{bcirc}(\mathcal{A-B})\right\|_2\\=&\|\mathcal{Q}^{-1}\|_2\cdot\left\|\mathcal{Q}\right\|_2\cdot\left\|\mathcal{A-B}\right\|_2,
\end{aligned}
\end{equation*}
%where we exploit the definition of the tensor spectral norm.
completing the proof of the theorem.
\end{proof}

\hskip 2em Next we derive the tensor Hoffman-Wielandt theorem.
\begin{theorem}%(Tensor Hoffman-Wielandt theorem)
Let $\mathcal{A}$ and $\mathcal{B}\in\mathbb{C}^{n\times n\times n_3}$ be both normal, with $\lambda_1,\ldots,\lambda_{nn_3}$ being the t-eigenvalues of $\mathcal{A}$, and assume that $\mu_1,\ldots,\mu_{nn_3}$ are the t-eigenvalues of $\mathcal{B}$. There is a permutation $\pi(\cdot)$ of the integers $1,\ldots,nn_3$ such that
\begin{equation*}
\left(\sum\limits_{i=1}^{nn_3}|\mu_{\pi(i)}-\lambda_i|^2\right)^{\frac{1}{2}}\le
{n_3}\left\|\mathcal{B}-\mathcal{A}\right\|_F.
\end{equation*}
\end{theorem}
$Proof.$
It is clear that $\mathtt{bcirc}(\mathcal{A})$ and $\mathtt{bcirc}(\mathcal{B})$ are both normal matrices. Consequently, there is a permutation $\pi(\cdot)$ of the integers $1,\ldots,nn_3$ such that
\begin{equation*}
	\left(\sum\limits_{i=1}^{nn_3}|\mu_{\pi(i)}-\lambda_i|^2\right)^{\frac{1}{2}}\le\left\|\mathtt{bcirc}(\mathcal{B})-\mathtt{bcirc}(\mathcal{A})\right\|_F
\end{equation*}
\begin{equation*}
=\left\|\mathtt{bcirc}(\mathcal{B}-\mathcal{A})\right\|_F=
{n_3}\left\|\mathcal{B}-\mathcal{A}\right\|_F. \xqedhere{118.5pt}{\qed}
\end{equation*}
%\begin{equation*}
%\begin{aligned}
%&\left(\sum\limits_{i=1}^{nn_3}|\mu_{\pi(i)}-\lambda_i|^2\right)^{\frac{1}{2}}\le\left\|\mathtt{bcirc}(\mathcal{B})-\mathtt{bcirc}(\mathcal{A})\right\|_F\\=&
%\left\|\mathtt{bcirc}(\mathcal{B}-\mathcal{A})\right\|_F=
%{n_3}\left\|\mathcal{B}-\mathcal{A}\right\|_F.
%\end{aligned} \xqedhere{118.5pt}{\qed}
%\end{equation*}

The tensor Hoffman-Wielandt theorem leads to the following corollary.
\begin{Corollary}
Soppose that $\mathcal{A}$ and $\mathcal{B}\in\mathbb{R}^{n\times n\times n_3}$ are symmetric, where $\lambda_1\le\ldots\le\lambda_{nn_3}$ are the t-eigenvalues of $\mathcal{A}$, and $\mu_1\le\ldots\le\mu_{nn_3}$ are the t-eigenvalues of $\mathcal{B}$. Then
\begin{equation*}
\left(\sum\limits_{i=1}^{nn_3}|\mu_{i}-\lambda_i|^2\right)^{\frac{1}{2}}\le{{n_3}}\left\|\mathcal{B}-\mathcal{A}\right\|_F.
\end{equation*}
\end{Corollary}

\hskip 2em Similar treatment yields the last theorem.
\begin{theorem}
Let $\mathcal{T}=\mathcal{A} + \mathtt i\mathcal{B}\in\mathbb{C}^{n\times n\times n_3}$ with real symmetric tensors $\mathcal{A}$ and $\mathcal{B}$. Let $\alpha_i$ and $\beta_i$ be the t-eigenvalues of $\mathcal{A}$ and $\mathcal{B}$ respectively, which are ordered such that $|\alpha_1|\ge|\alpha_2|\ge|\alpha_{nn_3}|$ and $|\beta_1|\ge|\beta_2|\ge|\beta_{nn_3}|$. Then
\begin{equation*}
\frac{1}{n_3}\left\|\mathtt{diag}(\alpha_1+ \mathtt i\beta_1,\ldots,\alpha_{nn_3} + \mathtt i\beta_{nn_3})\right\|_F\le\sqrt{2}\left\|\mathcal{T}\right\|_F
\end{equation*}
and
\begin{equation*}
	\left\|\mathtt{diag}(\alpha_1+\mathtt i\beta_1,\ldots,\alpha_{nn_3}+\mathtt i\beta_{nn_3})\right\|\le\sqrt{2}\left\|_2\mathcal{T}\right\|_2.
\end{equation*}
\end{theorem}
%\begin{proof}
%Utilizing the fact of \cite[Theorem 4.18]{Zhan2002} that
%\begin{equation*}
%\begin{aligned}
%&\left\|\mathtt{diag}(\alpha_1+\mathtt i\beta_1,\ldots,\alpha_{nn_3}+\mathtt i\beta_{nn_3})\right\|_F\le\sqrt{2}\left\|\mathtt{bcirc}(\mathcal{T})\right\|_F=n_3\sqrt{2}\left\|\mathcal{T}\right\|_F,
%\end{aligned}
%\end{equation*}
%the proof of the first inequality is completed.
%\end{proof}
\section{Conclusion}
\hskip 2em We have provided some inequalities for a variety of tensor topics.
Most of the results and proofs presented here are derived through the technique of unfolding tensors into block circulant matrices.
Several directions can be pursued to expand the results throughout this paper.
%\section*{Acknowledgments}
\bibliographystyle{siam}
\bibliography{TensorInequality}

\end{document}